\documentclass[reqno]{amsart}
\usepackage{amsfonts}
\usepackage{amsmath}
\usepackage{amssymb}
\usepackage{amsthm}
\usepackage{colonequals}
\usepackage{color}
\usepackage{enumerate}
\usepackage[margin=1in]{geometry}
\usepackage{tikz}
\DeclareMathOperator{\lip}{lip}
\DeclareMathOperator{\reg}{reg}
\DeclareMathOperator{\vecspan}{span}
\def\B{\mathbb{B}}
\def\R{\mathbb{R}}
\def\cov{{\rm cov}\,}
\def\gph{{\rm gph}\,}
\def\inte{{\rm int}\,}
\def\ker{{\rm ker}\,}
\def\emp{\emptyset}
\def\ox{\bar x}
\def\oy{\bar y}
\def\oz{\bar z}
\newtheorem{theorem}{Theorem}[section]
\newtheorem{remark}[theorem]{Remark}
\newtheorem{corollary}[theorem]{Corollary}
\newtheorem{definition}[theorem]{Definition}
\newtheorem{example}[theorem]{Example}

\newcommand{\da}{\substack{\rightarrow\\[-1em]\rightarrow}}

\begin{document}
\begin{abstract}
It is well known from the seminal Brockett's theorem that the openness property of the mapping on the right-hand side of a given nonlinear ODE control system is a necessary condition for the existence of locally asymptotically stabilizing continuous stationary feedback laws. However, this condition fails to be sufficient for such a feedback stabilization. In this paper we develop an approach of variational analysis to continuous feedback stabilization of nonlinear control systems with replacing openness by the {\em linear openness} property, which has been well understood and characterized in variational theory. It allows us, in particular, to obtain efficient conditions via the system data supporting the sufficiency in Brockett's theorem and ensuring local exponential stabilization by means of continuous stationary feedback laws. Furthermore, we derive new necessary conditions for local exponential and asymptotic stabilization of continuous-time control systems by using both continuous and smooth stationary feedback laws and establish also some counterparts of the obtained sufficient conditions for local asymptotic stabilization by continuous  stationary feedback laws in the case of nonlinear discrete-time control systems.
\end{abstract}

\title[]{Linear Openness and Feedback Stabilization\\of Nonlinear Control Systems}			
				
\author{Rohit Gupta}
\address{Institute for Mathematics and its Applications, University of Minnesota, Minneapolis, MN 55455}
\email{gupta311@umn.edu}

\author{Farhad Jafari}
\address{Department of Mathematics, University of Wyoming, Laramie, WY 82071}
\email{fjafari@uwyo.edu}

\author{Robert J. Kipka}
\address{Department of Mathematical Sciences, Kent State University at Stark, North Canton, OH 44720}
\email{rkipka@kent.edu}

\author{Boris S. Mordukhovich}
\address{Department of Mathematics, Wayne State University, Detroit, MI 48202 and RUDN University, Moscow 117198, Russian Federation}
\email{boris@math.wayne.edu}

\maketitle

\section{Introduction}
Consider the autonomous control system governed by nonlinear ordinary differential equations
\begin{eqnarray}
\dot{x}=f(x,u),\;t\ge 0,\label{cs}
\end{eqnarray}
where the vector function $f\colon U\times V\to\mathbb{R}^{n}$ on the right-hand side of (\ref{cs}) is sufficiently smooth on some open set $U\times V\subset\mathbb{R}^{n}\times\mathbb{R}^{m}$ around the origin $(0,0)$ as a given {\em equilibrium} pair. The continuous feedback stabilization property of our consideration is formulated as follows; see, e.g., \cite[Definition~10.11]{coron07}.

\begin{definition}\label{stab-def}
We say that the control system \eqref{cs} is {\sc locally asymptotically stabilizable} by means of {\sc continuous stationary feedback laws} if there exists a continuous control $u\in C(\mathbb{R}^{n},\mathbb{R}^{m})$ with $u(0)=0$ such that $0\in\R^n$ is a locally asymptotically stable equilibrium point for the closed-loop system
\begin{eqnarray}\label{closed-loop}
\dot{x}=f\big(x,u(x)\big).
\end{eqnarray}
\end{definition}

The question of whether \eqref{cs} can be locally asymptotically stabilized by means of continuous stationary feedback laws has been intensively studied in the literature with deriving some necessary as well as sufficient conditions for it; see, e.g., \cite{artstein83,brockett83,byrnes08,coron90,coron07,hermes91,sus94,sontag98,sontag99} and the references therein for a variety of results, discussions, and applications. In particular, the fundamental theorem by Brockett \cite{brockett83} shows that the control system (\ref{cs}) cannot be locally asymptotically stabilized by continuous stationary feedback laws if the {\em openness} property of $f$ is not satisfied at the equilibrium point $(0,0)$. In fact, Brockett's original result of the topological (mainly degree theoretic) nature \cite[Theorem~1]{brockett83} establishes the necessity of the openness property for the existence of a continuously differentiable (smooth) locally stabilizing feedback law. Its extension to merely continuous stationary feedback laws is given by Zabczyk \cite{zab89} by using a different device based on a deep result by Krasnoselskii and Zabreiko \cite{krasn83} from geometric nonlinear analysis. On the other hand, Coron \cite{coron90} and Sontag \cite{sontag99} construct impressive examples demonstrating that the openness property is not sufficient for the existence of such a locally asymptotically stabilizing continuous stationary feedback law.

In this paper we suggest to replace the openness property of $f$ by its {\em linear openness} counterpart, which is a basic property of single-valued and set-valued mappings that has been well understood and completely characterized (contrary to the openness property) in general settings of {\em variational analysis}; see Section~\ref{sec1a} for a brief overview and the books \cite{bor-zhu05,don-roc14,mor06a,mor06b,roc-wets98} with the extended bibliographies therein for numerous applications to various issues in optimization, equilibria, control, and practical models.

We believe that the developed approach of variational analysis and generalized differentiation has the strong potential for applications to a broader class of control systems, not just to smooth system \eqref{cs}, and discuss these issues in the concluding section. However, for the better understanding we confine ourselves here to the frameworks of smooth ODE systems of type \eqref{cs} and their discrete-time analogs. In the case of \eqref{cs} the developed variational approach allows us to verify that the linear openness of $f$ around the equilibrium point, being combined with a precise relationship between the exact bound of linear openness and unstable eigenvalues of the partial Jacobian matrix of $f$ in $x$, ensures the existence of a continuous stationary feedback law, which locally {\em exponentially} (and hence locally asymptotically) stabilizes \eqref{cs} via the closed-loop control system \eqref{closed-loop}. In this way we arrive at a partial converse to Brockett's theorem. Moreover, we justify the necessity of linear openness for local exponential stabilization of \eqref{cs} by means of continuous stationary feedback laws as well as the necessity of this property for local asymptotic stabilization of \eqref{cs} by means of smooth feedback laws under an additional assumption. Certain analogs of the obtained sufficient conditions for local asymptotic stabilization by means of continuous stationary feedback laws are establish also for the discrete-time nonlinear control counterpart of \eqref{cs}.

The rest of the paper is organized as follows. Section~\ref{sec1a} presents a brief overview of basic tools and results of variational analysis mainly concerning the linear openness property of single-valued and set-valued mappings as well as closely related properties of metric regularity and Lipschitzian stability, from both qualitative and quantitative viewpoints. The presented material justifies our variational approach to stabilizing nonlinear control systems not only of the smooth type \eqref{cs} and its discrete-time counterpart, but also for more general classes to be considered in future developments.

Section~\ref{sec2} is devoted to the implementation of this variational approach to establish sufficient as well as necessary conditions for stabilizing nonlinear continuous-time control systems \eqref{cs} by means of both continuous and smooth feedback laws. In Section~\ref{sec:disc} we deal with discrete-time control systems, Section~\ref{sec3} presents some illustrative examples, and the final Section~\ref{sec4} summaries the main achievements of the paper and discusses some perspectives of the variational approach for further research in this direction.

Throughout the paper we use standard notation of variational analysis and control theory; see, e.g., \cite{coron07,mor06a,roc-wets98}. Recall that $\mathbb{B}_{r}(z)$ stands for the closed ball centered at $z$ with radius $r>0$, while the symbol $\B$ signifies the closed unit ball of the space in question.

\section{Linear Openness and Related Properties of Nonlinear Mappings}\label{sec1a}

We start with recalling some properties of mappings used in the sequel. Most of the definitions below are valid or can be reformulated in general frameworks of normed and even metric spaces, but we need them in finite dimensions and so proceed accordingly.

\begin{definition}\label{open}
A single-valued mapping $f\colon\mathbb{R}^{l}\to\mathbb{R}^{n}$ is said to be {\sc open} at $\bar z\in\R^l$ if the $f$-image of every neighborhood of $\oz$ contains/covers a neighborhood of $f(\oz)$. This amounts to saying that
\begin{eqnarray}\label{open-prop}
f(\oz)\in\inte f(U)\;\mbox{ for any neighborhood }\;U\;\mbox{ of }\;\oz.
\end{eqnarray}
\end{definition}

The origin of this property goes back to the classical Banach-Schauder {\em open mapping} theorem saying that a bounded linear operator between Banach spaces is open if and only if it is surjective. The openness property \eqref{open-prop} of nonlinear vector functions $f$ was used by Brockett \cite{brockett83} as a necessary condition for the existence of locally asymptotically stabilizing continuous stationary feedback laws in the sense of Definition~\ref{stab-def} for the continuous-time control systems \eqref{cs}. We now recall an appropriate strengthening of \eqref{open-prop}, which has been recognized as a fundamental property in variational analysis for nonlinear (and even set-valued) mappings while allowing us to get the sufficiency in Brockett's theorem.

\begin{definition}\label{lin-open}
A mapping $f\colon\mathbb{R}^{l}\to\mathbb{R}^{n}$ is said to be {\sc linearly open} $($or having the {\sc covering property}$)$ around $\bar z$ with modulus $\kappa>0$ if there exists neighborhood $U$ of $\bar{z}$ such that
\begin{align}
\mathbb{B}_{\kappa r}\big(f(z)\big)\subset f\big(\mathbb{B}_{r}(z)\big)\;\mbox{ for any }\;z\in U\;\mbox{ and }\;r>0\;\mbox{ with }\;\mathbb{B}_{r}(z)\subset U.\label{cov}
\end{align}
The supremum of all the moduli $\{\kappa\}$ for which \eqref{cov} holds with some neighborhood $U$ is called the {\sc exact covering/linear openness bound} of $f$ around $\oz$ and is denoted by $\cov f(\oz)$.
\end{definition}

Property \eqref{cov} first appeared in \cite{dmo80} under the name of ``covering in a neighborhood'' while it has been introduced and popularized by Milyutin in his talks and personal communications a long time before the publication of \cite{dmo80}. The interest to designating and calculating the (quantitative) {\em exact bound} $\cov f(\oz)$ arose later on; see \cite{mor93,mor06a,roc-wets98} and the references therein. The term ``linear openness'' for \eqref{cov} was first used probably in \cite{roc-wets98}; this property is also widespread in the literature under the name of ``openness at linear rate."

It is clear that the linear openness property from Definition~\ref{lin-open} yields its openness counterpart from Definition~\ref{open}. Simple examples demonstrate that the opposite implication fails.

\begin{example}\label{exam-open}
{\rm Consider the real-valued function $f\colon\R\to\R$ given by $f(z)\colonequals z^3$. It is easy to check that $f$ possesses the openness property \eqref{open-prop} at $\oz=0$. However, the linear openness property \eqref{cov} obviously fails for $f$ around the same point. The reader can observe that the same phenomenon holds true for the power functions $f(z)\colonequals z^s$ with any odd natural exponent $s\ge 3$ with $\oz=0$.}
\end{example}

It is important to emphasize the following two principal differences between the openness property \eqref{open-prop} and its linear counterpart \eqref{cov} from Definition~\ref{lin-open}:
\begin{enumerate}[(i)]
\item{Property \eqref{cov} ensures the {\em uniformity} of covering {\em around} $\oz$.}
\item{Property \eqref{cov} designates a {\em linear rate} of openness {\em quantified} by modulus $\kappa$.}
\end{enumerate}

These two issues are crucial to establish complete characterizations of linear openness/covering with a precise calculation of the exact covering bound $\cov f(\oz)$. Before proceeding in this direction, let us recall another fundamental property of variational analysis, which happens to be equivalent to linear openness with the reciprocal relationship between the exact bounds of the corresponding moduli.

\begin{definition}\label{metreg}
A mapping $f\colon\mathbb{R}^{l}\to\mathbb{R}^{n}$ is said to be {\sc metrically regular} around $\bar z$ with modulus $\mu>0$ if there exist neighborhoods $U$ of $\oz$ and $V$ of $\bar{y}\colonequals f(\oz)$ such that the distance estimate
\begin{align}
d\big(z,f^{-1}(y)\big)\le\mu\|y-f(z)\|\;\mbox{ for any }\;z\in U\;\mbox{ and }\;y\in V\label{eqdef1}
\end{align}
holds via the distance function $d(z,\Omega)\colonequals\underset{u\in\Omega}{\inf}\;\|z-u\|$ for a given set $\Omega$. The infimum of all the moduli $\{\mu\}$ for which \eqref{eqdef1} is satisfied along with some neighborhoods $U$ and $V$ is called the {\sc exact regularity bound} of $f$ around $\oz$ and is denoted by $\reg f(\oz)$. Thus the absence of metric regularity is signaled by $\reg f(\bar{z})=\infty$.
\end{definition}

We refer the reader to the books \cite{bor-zhu05,don-roc14,mor06a,roc-wets98} and the bibliographies therein for the genesis of this property, its characterizations, various relationships, and numerous applications. In particular, it has been realized that metric regularity and covering properties are actually equivalent to each other and their exact bounds are reciprocally related as
\begin{eqnarray}\label{reg-cov}
\cov f(\oz)\cdot\reg f(\oz)=1.
\end{eqnarray}
They can also be equivalently described by {\em Lipschitzian} behavior of the (set-valued) inverse $f^{-1}\colon\mathbb{R}^{n}\da\mathbb{R}^{l}$ around $(f(\oz),\oz)$. Such a property is defined for arbitrary set-valued mappings as follows.

\begin{definition}\label{lip}
A set-valued mapping $F\colon\mathbb{R}^{l}\da\mathbb{R}^{n}$ is said to be {\sc Lipschitz-like} with modulus $\ell\ge 0$ around the pair $(\oz,\oy)$ belonging to the graph
$$
\gph F\colonequals\big\{(z,y)\in\mathbb{R}^{l}\times\mathbb{R}^{n}\big|\;y\in F(z)\big\}
$$
if there exist neighborhoods $U$ of $\oz$ and $V$ of $\oy$ such that
\begin{equation}\label{lip-l}
F(z)\cap V\subset F(u)+\ell\|z-u\|\B\;\mbox{ for all }\;z,u\in U.
\end{equation}
The infimum of all the moduli $\{\ell\}$ for which {\rm(\ref{lip-l})} holds with some neighborhoods $U$ and $V$ is called the {\sc exact Lipschitzian bound} of $F$ around $(\oz,\oy)$ and is denoted by $\lip F(\oz,\oy)$.
\end{definition}

If $F$ is single-valued, the Lipschitz-like property from Definition~\ref{lip} goes back to the classical local Lipschitzian behavior, while in the compact-valued case with $V=\R^n$ in \eqref{lip-l} it reduces to the Hausdorff local Lipschitzian property of multifunctions. In the general case of $V$ in (\ref{lip-l}) this condition is also known as the pseudo-Lipschitz or Aubin property. It can viewed as a natural graphical localization of Lipschitzian behavior for set-valued mappings. Note that in the case of single-valued mappings $F=f$ the exact Lipschitzian bound $\lip f(\oz)$ of $f$ around $\oz$ can be easily represented in the following way:
$$
\lip f(\bar{z})=\limsup_{z,u\to\bar{z}}\frac{\|f(z)-f(u)\|}{\|z-u\|}.
$$

It is known that the linear openness and metric regularity properties of $f$ around $\oz$ from Definition~\ref{lin-open} and Definition~\ref{metreg}, respectively, are equivalent to the Lipschitz-like property of Definition~\ref{lip} for the inverse mapping $F=f^{-1}$ around $(f(\oz),\oz)$ with the exact bound relationship
$$
\lip f^{-1}(\oy,\oz)=\reg f(\oz)\;\mbox{ where }\;\oy=f(\oz).
$$
Furthermore, the aforementioned equivalences and exact bound relationships hold true for general set-valued mappings between Banach spaces with appropriately extended definitions of covering and metric regularity; see \cite[Subsections~1.2.2 and 1.2.3]{mor06a} and the commentaries therein for a complete account.

When $f$ is of class $C^1$ as in our primary setting here, it has been shown independently by Lyusternik \cite{lus34} and Graves \cite{gra50} that the surjectively of the derivative operator $\nabla f(\oz)$, which amounts to the full rank condition for the Jacobian matrix $\nabla f(\oz)$ in finite dimensions, is sufficient for the openness property \eqref{open-prop} in \cite{gra50} and for a version of metric regularity (as a description of the tangent space to a smooth manifold) in \cite{lus34}. Note that neither Lyusternik nor Graves dealt with the linear openness and metric regularity properties defined above but their (rather close to each other) proofs were instrumental to verify the {\em sufficiency} of the surjectivity condition on $\nabla f(\oz)$ for the more delicate properties from Definitions~\ref{lin-open} and \ref{metreg} in the case of smooth single-valued mappings between arbitrary Banach spaces. However, they did not consider either the necessity of the surjectivity condition for the underlying properties, or the exact bound calculations.

It is implied by more general results of variational analysis (and has been first observed in its framework) that we actually have {\em complete qualitative and quantitative characterizations} of the basic properties defined above. The following theorem is taken from \cite[Theorem~1.57]{mor06a}, while being valid in general Banach spaces; see also the commentaries to \cite[Chapter~1]{mor06a} for further discussions. Recall that $f\colon\R^l\to\R^n$ is {\em strictly differentiable} at $\oz$ with the strict derivative $\nabla f(\oz)\colon\R^l\to\R^n$ (or the Jacobian matrix $\nabla f(\oz)$  in the finite-dimensional setting under consideration) if
\begin{eqnarray*}
\lim_{z,u\to\oz}\frac{f(z)-f(u)-\nabla f(\oz)(z-u)}{\|z-u\|}=0.
\end{eqnarray*}
This property always holds if $f$ is smooth (i.e., of class $C^1$) around $\oz$ but may be stronger than the standard Fr\'echet differentiability of $f$ at this point.

\begin{theorem}\label{lg}
Let $f\colon\R^l\to\R^n$ with $n\le l$ be strictly differentiable at $\oz$. Then $f$ enjoys the equivalent linear openness and metric regularity properties around $\oz$ if and only if the strict derivative $\nabla f(\oz)$ is surjective, i.e., the Jacobian matrix has full rank. Furthermore, the exact bounds of the linear openness and metric regularity of $f$ at $\oz$ are precisely calculated, respectively, by
\begin{eqnarray}\label{cov-mod}
\cov f(\oz)=\min\big\{\|\nabla f(\oz)^*v\|\;\big|\;\|v\|=1\big\}\;\mbox{ and }\;\reg f(\oz)=\big\|\big(\nabla f(\oz)^*\big)^{-1}\big\|,
\end{eqnarray}
where the symbol ``$^*$'' stands for the adjoint operator $($matrix transposition$)$, and where the last norm in \eqref{cov-mod} is the standard norm for linear bounded operators applied to the inverse operator $(\nabla f(\oz)^*)^{-1}$, which is single-valued in the case of the surjectivity of the derivative operator $\nabla f(\oz)$.
\end{theorem}

Having in mind further possible extensions (see more in Section~\ref{sec4}) of our approach and results to {\em nonsmooth} control systems \eqref{cs} and their {\em set-valued} versions $\dot{x}\in F(x)$ as well as to partial differential counterparts, we formulate now for the reader's convenience generalized differential characterizations of the aforementioned basic properties of variational analysis in the general framework of set-valued mappings $F\colon\R^l\da\R^n$ with closed graphs. This is done by using the {\em coderivative} $D^*F(\oz,\oy)\colon\R^{n}\da\R^l$ of $F$ at $(\oz,\oy)\in\gph F$, which we are not going to define here while recalling that
\begin{eqnarray}\label{cod-sm}
D^*f(\oz)(v)=\big\{\nabla f(\oz)^*v\}\;\mbox{ for all }\;v\in\R^n
\end{eqnarray}
provided that the mapping $F=f$ is single-valued  and strictly differentiable at $\ox$; in particular,  if it is of class $C^1$ as in the case studied in this paper. In general, the coderivative is a positively homogeneous mapping that enjoys comprehensive {\em calculus rules} based on {\em variational/extremal principles} of variational analysis. We refer the reader to the books \cite{mor06a,mor06b,roc-wets98} and the commentaries therein for exact results, discussions, extended bibliographies,  and various applications.

In particular, a complete characterization of the Lipschitzian property from Definition~\ref{lip} for $F$ around $(\oz,\oy)$ with the precise calculation of the exact bound $\lip F(\oz,\oy)$ obtained in \cite[Theorem~5.7]{mor93} reads as
\begin{eqnarray}\label{crit}
D^*F(\oz,\oy)(0)=\big\{0\big\}\;\mbox{ with }\;\lip F(\oz,\oy)=\big\|D^*F(\oz,\oy)\big\|,
\end{eqnarray}
where $\|\cdot\|$ stands for the usual norm of a positively homogeneous mapping $G\colon\R^n\da\R^l$ defined by
$$
\|G\|\colonequals\sup\big\{\|u\|\;\big|\;u\in G(v),\;\|v\|\le 1\big\}.
$$
Another proof of \eqref{crit} is given in \cite[Theorem~9.40]{roc-wets98} under the name of ``Mordukhovich criterion." The reader can find an infinite-dimensional extension of \eqref{crit} in \cite[Theorem~4.10]{mor06a}. The corresponding counterparts of \eqref{crit} for the (inverse) equivalent properties of metric regularity and linear openness are
\begin{eqnarray}\label{reg-crit}
\begin{array}{ll}
\ker D^*F(\oz,\oy)=\big\{0\big\}\;\mbox{ with }\;\reg F(\oz,\oy)=\|D^*F(\oz,\oy)^{-1}\|=\|D^*F^{-1}(\oy,\oz)\|,\\\\
\qquad\qquad\cov F(\oz,\oy)=\inf\big\{\|u\|\;\big|\;u\in D^*F(\oz,\oy)(v),\;\|v\|=1\big\},
\end{array}
\end{eqnarray}
where $\ker G\colonequals\{v\in\R^n|\;0\in G(v)\}$. Due to the coderivative representation \eqref{cod-sm} for strictly differentiable single-valued mappings, the kernel coderivative criterion for metric regularity and linear openness in \eqref{reg-crit} reduces to the surjectivity of $\nabla f(\oz)$ while the exact bound formulas in \eqref{reg-crit} get back to \eqref{cov-mod} for such mappings as stated in Theorem~\ref{lg} above.

After the given brief overview of available tools of variational analysis related to linear openness, we are now ready to proceed with the implementation of the variational approach in the case of smooth continuous-time control systems \eqref{cs} as well for the corresponding ones with discrete time.

\section{Exponential and Asymptotic Stabilization of Continuous-Time Control Systems}\label{sec2}

Consider first the continuous-time control system \eqref{cs} and recall the notion of its {\em local exponential stabilization} by using continuous stationary feedback laws. The main attention here is paid to the following stabilization notion for the nonlinear control system \eqref{cs}.

\begin{definition}\label{stabil}
The control system \eqref{cs} is said to be {\sc locally exponentially stabilizable} if there exist a continuous stationary feedback law $u(x)$ with $u(0)=0$ and constants $\alpha>0$, $M>0$, and $\delta>0$ such that for any starting point $x_0\in\R^n$ with $\|x_0\|<\delta$ there exists a unique solution $x(t,x_0)$ of the closed-loop control system \eqref{closed-loop} satisfying $x(0,x_0)=x_0$ and the exponential decay condition
\begin{eqnarray*}
\|x(t,x_0)\|\le M e^{-at}\|x_0\|\;\mbox{ whenever }\;t\ge 0.
\end{eqnarray*}
\end{definition}

To formulate the following main result of the paper, consider the partial Jacobian matrices
\begin{eqnarray}\label{A}
A\colonequals\nabla_x f(0,0)\;\mbox{ and }\;B\colonequals\nabla_u f(0,0),
\end{eqnarray}
for $f\in C^1$, the (complex) {\em spectrum} $\Lambda(A)$ of $A$, and the collection of {\em eigenvalues with nonnegative real parts}
\begin{eqnarray}\label{lamb+}
\Lambda_+(A)\colonequals\big\{\lambda\in\Lambda(A)\big|\;\lambda\in\mathbb{C}\;\mbox{ \textnormal{with} }\;{\rm Re}(\lambda)\ge 0\big\}
\end{eqnarray}
of $A$ in the case of the continuous-time control system \eqref{cs}.

\begin{theorem}\label{th1}
Let $f\colon\R^n\times\R^m\to\R^n$ be of class $C^{1}$ in a neighborhood of the equilibrium pair $(0,0)$, let the matrices $A$ and $B$ be defined in \eqref{A}, and let $\Lambda_+(A)\subset\R$. Assume in addition that:

{\bf(C)} $f$ is linearly open around $(0,0)$ with $\cov f(0,0)>\kappa$ for some positive constant $\kappa$ satisfying
\begin{eqnarray}\label{kappa}
\kappa>\eta_c\colonequals\sup_{\lambda\in\Lambda_+(A)}\lambda,
\end{eqnarray}
where $\sup\emp\colonequals-\infty$. Then the control system \eqref{cs} can be locally exponentially stabilized by means of continuous stationary feedback laws. Furthermore, the linear openness of $f$ around $(0,0)$ is necessary for such a local exponential stabilization of \eqref{cs} without any additional assumptions on the system data if continuously differentiable stationary feedback laws are used instead of merely continuous ones.
\end{theorem}

\begin{proof}
The verifications of both sufficient and necessary parts of the theorem strongly involve the usage of the complete linear openness characterization given in Theorem~\ref{lg}.

We start with proving the {\em sufficiency} statement of this theorem. Besides Theorem~\ref{lg}, it employs a stability property for linear openness with respect to small Lipschitzian perturbations along with some results well understood in controllability and stabilization theories for smooth control systems. To proceed in this way, for any constant $\nu>0$ we consider the perturbed vector function
\begin{eqnarray*}
f_{\nu}(x,u)\colonequals f(x,u)+g_{\nu}(x,u)\;\mbox{ with }\;g_{\nu}(x,u)\colonequals-\nu x.
\end{eqnarray*}
It is obvious that $\lip g_{\nu}(0,0)=\nu$. Taking into account the equivalence between the linear openness and metric regularity properties discussed Section~\ref{sec1a} with the exact bound relationship \eqref{reg-cov}, we deduce from the linear openness version of \cite[Theorem~4.25]{mor06a} that
\begin{eqnarray}\label{cov-est}
\cov f_\nu(0,0)>\kappa-\nu>0\;\mbox{ whenever }\;\nu\in[0,\kappa).
\end{eqnarray}
Thus the derivative operator $\nabla f_{\nu}|_{(0,0)}\colon\R^n\times\R^m\to\R^n$ is surjective for such $\nu$ due to the {\em necessity} of the surjectivity (i.e., full rank of the Jacobian matrix $\nabla f_{\nu}(0,0)$) condition for the linear openness property from Theorem~\ref{lg}. Furthermore, condition \eqref{kappa} in assumption {\bf(C)} implies that
\begin{eqnarray*}
\cov f_{\lambda}(0,0)>\kappa-\lambda>0\;\mbox{ for all }\;\lambda\in\Lambda_+(A),
\end{eqnarray*}
and hence the operator $\nabla f_{\lambda}|_{(0,0)}\colon\R^n\times\R^m\to\R^n$ is surjective for all $\lambda\in\Lambda_+(A)$ by Theorem~\ref{lg}. The latter condition can be equivalently written in the rank form
\begin{align}
\text{rank}\;\big[A-\lambda I\mid B\big]=n\;\mbox{ for all }\;\lambda\in\Lambda_+(A).\label{eqpf1}
\end{align}
Since \eqref{eqpf1} is the well-known {\em Hautus test} for {\em asymptotic controllability} of linear autonomous systems \cite{hautus70,sontag98}, we have that the linearized control system
\begin{eqnarray}\label{lineari}
\dot x=Ax+Bu,\;t\ge 0,
\end{eqnarray}
with $A$ and $B$ from \eqref{A}, is asymptotically controllable. Employing now \cite[Theorem~4]{hautus70} and its proof for the case of continuous-time  control systems allows us to conclude that the linearized system \eqref{lineari} is exponentially stabilizable, which in turn yields as in \cite[Chapter~6; see, e.g., Exercise~6 on p.\ 392]{lee-mar67} the local exponential stabilization of the original nonlinear system \eqref{cs} and hence verifies the sufficiency part of the theorem.

To justify next the remaining {\em necessity} statement of this theorem, we assume that the nonlinear control system \eqref{cs} is exponentially stabilizable by means of {\em smooth} stationary feedback laws. It is equivalent to the exponential stabilization of the linearized system \eqref{lineari} (see, e.g., \cite[Propositions~1 and 2]{zab89}) with the possibility to use even linear feedback laws therein. The latter implies, by the equivalence between (i) and (iii) in \cite[Theorem~4]{hautus70} for continuous-time control systems, that the linearized system \eqref{lineari} is asymptotically controllable. Thus the Hautus test for asymptotic controllability of linear autonomous continuous-time control systems tells us, due to the equivalence between (i) and (ii) in \cite[Theorem~4]{hautus70}, that
\begin{eqnarray}\label{haut-c}
\text{rank}\;\big[A-\lambda I\mid B\big]=n\;\mbox{ for all complex numbers }\;\lambda\;\mbox{ with }\;\mbox{Re}(\lambda)\ge 0;
\end{eqnarray}
see, e.g., \cite[Exercise~5.5.7]{sontag98}. Plugging now $\lambda=0$ into \eqref{haut-c} gives us the simple rank condition
\begin{eqnarray}\label{eqrem2}
\textnormal{rank}\;\big[A\mid B\big]=n,
\end{eqnarray}
which ensures, by the {\em sufficiency part} of the linear openness characterization in Theorem~\ref{lg}, that $f$ is linearly open around $(0,0)$. This completes the proof of the theorem.
\end{proof}

\begin{remark}\label{kalman}
{\rm We get from the characterization of linear openness presented in Theorem~\ref{lg} that the vector function $f$ from Theorem~\ref{th1} enjoys this property if and only if the rank condition \eqref{eqrem2} holds, which is surely different from the classical Kalman rank condition for controllability of linear autonomous continuous-time control systems; see, e.g., \cite{lee-mar67}. As follows from the above proof of Theorem~\ref{th1}, the additional assumptions in {\bf(C)} allow us to justify the validity of the Hautus test for asymptotic controllability \eqref{eqpf1} of the linearized system \eqref{lineari}. It is clear that the opposite implication does not hold, i.e., \eqref{eqrem2} does not imply \eqref{eqpf1}. Indeed, there are simple examples of linear control systems, which are not controllable while satisfying the underlying rank condition \eqref{eqrem2}; see, e.g., $\dot x_1=x_1,\;\dot x_2=u$.}
\end{remark}

The following two consequences of Theorem~\ref{th1} provide some specifications of this result for particular subclasses of continuous-time control \eqref{cs}. They are similar to those observed by Brockett \cite{brockett83} in the case of local asymptotic stabilization of \eqref{cs}.

\begin{corollary}\label{cor1}
Suppose that in the framework of Theorem~{\rm\ref{th1}} we have
\begin{eqnarray}\label{cor1-cs}
f(x,u)\colonequals g_{0}(x)+\sum_{i=1}^{m}g_{i}(x)u_{i},
\end{eqnarray}
where $\vecspan\{g_{0}(\cdot),\hdots,g_{m}(\cdot)\}=\R^{d}$. If $d<n$, then the control system \eqref{cs} cannot be locally exponentially stabilized by means of continuously differentiable stationary feedback laws.
\end{corollary}

\begin{proof}
Assume that \eqref{cs} is locally exponentially stabilizable by a smooth stationary feedback law. Then the necessity part of Theorem~\ref{th1} yields the validity of the rank condition \eqref{eqrem2}, i.e., the linear openness of $f$ around $(0,0)$ by Theorem~\ref{lg}. The conclusion of this corollary can be deduced similarly to Brockett's observation made in \cite[the first Remark on p.\ 187]{brockett83}.
\end{proof}

\begin{corollary}\label{cor2}
Suppose that in the framework of Theorem~{\rm\ref{th1}} we have
\begin{eqnarray*}
f(x,u)\colonequals\sum_{i=1}^{m}g_{i}(x)u_{i},
\end{eqnarray*}
where the vectors $\{g_{i}(0)\}_{i=1}^{m}$ are linearly independent. Then the control system \eqref{cs} can be locally exponentially stabilized by means of continuously differentiable stationary feedback laws if and only if $m=n$.
\end{corollary}

\begin{proof}
This follows from combining Theorem~\ref{th1} and Corollary~\ref{cor1}.
\end{proof}

\begin{remark}\label{rem-ind}
{\rm It is crucial to assume in Corollary~\ref{cor2} that the vectors $\{g_{i}(0)\}_{i=1}^{m}$ are linearly independent; cf. \cite{brockett83} and \cite{sontag99} for the case of local asymptotic stabilization.}
\end{remark}

Let us now present several consequences of the sufficiency in Theorem~\ref{th1}.

\begin{corollary}\label{spectrum}
Suppose that in the setting of Theorem~{\rm\ref{th1}} we have $\Lambda(A)\subset\R$ and that the number $\eta_c$ in assumption {\bf(C)} is replaced by the quantity
\begin{eqnarray*}
\Tilde\eta_c\colonequals\max_{i\in\{1,\hdots,n\}}\big\{|\lambda_{i}|\;\big|\;\lambda_{i}\in\Lambda(A)\big\}
\end{eqnarray*}
calculated along the entire spectrum $\Lambda(A)$ of the matrix $A$ instead of just its nonnegative real part \eqref{lamb+}. Then the control system \eqref{cs} can be locally exponentially stabilized by means of continuous stationary feedback laws.
\end{corollary}

\begin{proof}
This follows directly from Theorem~\ref{th1} and the parameter choice in assumption {\bf(C)} of the theorem and its counterpart imposed in this corollary. On the other hand, the conclusion of this corollary while regarding the (generally weaker) local {\em asymptotic stabilization} of \eqref{cs} in the sense of Definition~\ref{stab-def} can be derived from \cite[Theorem~10.14]{coron07} married to the linear openness arguments employed above and using the fact that the modified rank condition
\begin{eqnarray}\label{haut-cont}
\text{rank}\;\big[A-\lambda I\mid B\big]=n\;\mbox{ for all }\;\lambda\in\Lambda(A)
\end{eqnarray}
ensures the controllability of the linearized system $\dot{x}=Ax+Bu$ by the {\em Hautus test} for {\em controllability} from \cite[Theorem~1]{hautus70} applied to the linearized continuous-time control system \eqref{lineari}.
\end{proof}

\begin{remark}\label{symmetric}
{\rm Note that control systems in which $A$ is a {\em symmetric matrix} automatically satisfy the assumption that $\Lambda(A)\subset\mathbb{R}$. This surely implies that $\Lambda_{+}(A)\subset\R$ in the assumptions of Theorem~\ref{th1} and that $\Lambda_{1}(A)\subset\R$ in the corresponding result of Theorem~\ref{th4} for the discrete-time control systems \eqref{dtcs}.}
\end{remark}

It is easy to see that condition \eqref{eqpf1} in Theorem~\ref{th1} may be significantly better than the corresponding one \eqref{haut-cont} in the proof of Corollary~\ref{spectrum}. Note also that in the case where $\Lambda_+(A)=\emptyset$, assumption {\bf(C)} of Theorem~\ref{th1} reduces simply to the requirement that the mapping $f$ in \eqref{cs} is linearly open around the equilibrium point $(0,0)$, i.e., in this case necessary and sufficient conditions of Theorem~\ref{th1} merge.

\begin{remark}\label{rem1}
{\rm It is worth mentioning that the exact covering bound $\cov f(0,0)$ used in Theorem~\ref{th1} and Corollary~\ref{spectrum} is precisely calculated by formula \eqref{cov-mod} with $\oz=(0,0)\in\R^n\times\R^m$ therein. Thus all the conditions of these results can be described entirely in terms of the {\em initial data} of the control system \eqref{cs} without involving any rank calculations.}
\end{remark}

\begin{corollary}\label{cor1a}
Assume that in the setting of Theorem~{\rm\ref{th1}} we have the condition $\Lambda_+(A)=\{0\}$. Then the linear openness of $f$ around $(0,0)$ ensures that the control system \eqref{cs} can be locally exponentially stabilized by means of continuous stationary feedback laws.
\end{corollary}

\begin{proof}
It is obvious to see that condition {\bf(C)} of Theorem~\ref{th1} holds automatically under the assumptions of this corollary, and thus the conclusion follows.
\end{proof}

\begin{remark}\label{sontag} {\rm
Observe that the setting of Corollary~\ref{cor1a} encompasses an important case where
$$
f(x,u)\colonequals\sum_{i=1}^{n}g_{i}(x)u_{i}
$$
with the vectors $\{g_{i}(0)\}_{i=1}^{n}$ being linearly independent. It is worth noting that the latter setting covers the situation of the example presented in \cite{sontag99} showing that the sufficiency of the openness property in Brockett's theorem fails. Indeed, in Sontag's example all the assumptions of Corollary~\ref{cor1a} are satisfied {\em except} the linear openness of $f$ around the origin.}
\end{remark}

\begin{corollary}\label{th2}
Assume that in the framework of Theorem~{\rm\ref{th1}} we have
\begin{eqnarray}\label{lcs}
f(x,u)\colonequals Ax+Bu,
\end{eqnarray}
with $\Lambda_{+}(A)\subset\R$ and let $g\colon\R^n\times\R^m\to\R^n$ be of class $C^{1}$ around the origin with $g(0,0)=0$, $\nabla_x g(0,0)=0$, and $\nabla_u g(0,0)=0$. Then the validity of condition {\bf(C)} from Theorem~{\rm\ref{th1}} ensures that the $($slightly$)$ perturbed nonlinear control system
\begin{eqnarray}\label{per-cs}
\dot{x}=Ax+Bu+g(x,u)
\end{eqnarray}
can be locally exponentially stabilized by means of continuous stationary feedback laws.
\end{corollary}

\begin{proof}
This follows directly from Theorem~\ref{th1} applied to the perturbed control system \eqref{per-cs}.
\end{proof}

Having in mind these observations and the corresponding developments presented in \cite{coron07}, we arrive at the following two corollaries, which are consequences of Corollary~\ref{spectrum} and its proof. For the first one, the reader is referred to \cite[Definition~3.2]{coron07}, where the notion of {\em small-time local controllability} of \eqref{cs} is formulated.

\begin{corollary}\label{small time}
Let all the assumptions of Corollary~{\rm\ref{spectrum}} hold. Then the control system \eqref{cs} is small-time locally controllable at the equilibrium point $(0,0)$.
\end{corollary}

\begin{proof}
It is shown in \cite[Theorem~3.8]{coron07} that the controllability of the linearized control system \eqref{lineari} ensures the small-time local controllability of \eqref{cs}. On the other hand, the proof of Corollary~\ref{spectrum} verifies that the assumptions imposed therein yield controllability of the linearized control \eqref{lineari} via the Hautus test for controllability from \cite[Theorem~1]{hautus70}. Thus we get the claimed result.
\end{proof}

The next result shows that the linear openness and spectrum assumptions of Corollary~\ref{spectrum} for the case of linear autonomous continuous-time  control systems ensure actually {\em global controllability} in time $T$ (as defined, e.g., in \cite{coron07}) of not only the system itself, but also of its {\em bounded} perturbations.

\begin{corollary}\label{global}
Consider the linear version of \eqref{cs} with $f$ defined by \eqref{lcs}. Assume that $\Lambda(A)\subset\R$ and that $f$ is linearly open around $(0,0)$ with $\cov f(0,0)>\kappa$ for some constant $\kappa>0$ satisfying
\begin{eqnarray*}
\kappa>\Tilde\eta_c\colonequals\max_{\lambda\in\Lambda(A)}\lambda.
\end{eqnarray*}
Given in addition a continuous function $g\in C^1(\R^n,\R^n)$, suppose that it is bounded on $\R^n$. Then for each $T>0$ the boundedly perturbed control system
\begin{eqnarray}\label{bound}
\dot{x}=Ax+Bu+g(x)
\end{eqnarray}
is globally controllable in the fixed time $T$.
\end{corollary}

\begin{proof}
Similarly to the proof of Theorem~\ref{th1} we conclude that the assumptions imposed in this corollary ensure the validity of the Hautus test \eqref{haut-cont} for controllability of the linear control system under consideration. To arrive at the claimed conclusion, it remains to apply the result of \cite[Corollary~3.41]{coron07}.
\end{proof}

\begin{example}\label{bounded}
{\rm Consider the two-dimensional nonlinear control system from \cite[Exercise~3.42]{coron07}:
\begin{eqnarray*}
\dot x_1=x_2+g(x_1,x_2),\;\dot x_2=u\in\R,
\end{eqnarray*}
where $g(x_1,x_2)\colonequals-x_2$ in \eqref{bound} is unbounded. It is easy to check that this system is not globally controllable while all but the boundedness assumption of $g$ in Corollary~\ref{global} are satisfied.}
\end{example}

Finally in this section, we revisit the Brockett's classical setting \cite{brockett83} concerning {\em local asymptotic stabilization} of \eqref{cs} by means of {\em smooth} stationary feedback laws and derive the strengthening of his main (third) necessary condition (with the replacement of openness by {\em linear openness}) from the first one in \cite[Theorem~1]{brockett83} under an additional assumption on the collection of eigenvalues with nonnegative real parts.

\begin{theorem}\label{th3}
Let $f$, $A$, and $B$ be as in Theorem~{\rm\ref{th1}}, and assume that
\begin{eqnarray}\label{lambda+}
\Lambda_+(A)=\big\{\lambda\in\Lambda(A)\big|\;\lambda\in\mathbb{C}\;\;{\rm with }\;\;{\rm Re}(\lambda)>0\big\}
\end{eqnarray}
for the set $\Lambda_+(A)$ defined in \eqref{lamb+}. Then the local asymptotic stabilization of system \eqref{cs} by means of continuously differentiable stationary feedback laws implies that $f$ is linearly open around the equilibrium $(0,0)$.
\end{theorem}

\begin{proof}
Brockett's {\em first necessary condition} in \cite[Theorem~1(i)]{brockett83} tells us that the local asymptotic stabilization of the control system \eqref{cs} by means of smooth feedback laws implies that the linearized system \eqref{lineari} should have no uncontrollable modes associated with eigenvalues whose real part is positive. On the other hand, condition \eqref{lambda+} means that that there are no purely imaginary eigenvalues in the collections of eigenvalues with nonnegative real parts $\Lambda_+(A)$ of the matrix $A$. Taking into account the definition of $\Lambda_+(A)$ in \eqref{lamb+} and combining it with the imposed assumption \eqref{lambda+} ensure that the Hautus test \eqref{eqpf1} for asymptotic controllability of \eqref{lineari} holds. This readily yields, as in the proof of the necessity in Theorem~\ref{th1}, that the simple rank condition \eqref{eqrem2} is satisfied. The latter amounts to the surjectivity of the derivative operator $\nabla f|_{(0,0)}\colon\R^n\times\R^m\to\R^n$. Employing the sufficiency part of Theorem~\ref{lg} verifies the linear openness of $f$ around $(0,0)$ and thus completes the proof of the theorem.
\end{proof}

\begin{remark}\label{1cond}
{\rm Note that the aforementioned first necessary condition in \cite[Theorem~1]{brockett83} is different from the Hautus test for asymptotic controllability as, e.g., for the system
\begin{eqnarray*}
\dot{x}_{1}=u_{1}^{3},\;\dot{x}_{2}=u_{2}^{3}.
\end{eqnarray*}
Assumption \eqref{lambda+} imposed in Theorem~\ref{th3} merges these two conditions.}
\end{remark}

Similarly to Corollary~\ref{cor1} of Theorem~\ref{th1} in the case of local exponential stabilization, we now arrive at the corresponding consequence of Theorem~\ref{th3} for the local asymptotic stabilization of \eqref{cs}, where the form of $f$ in \eqref{cs} is given by \eqref{cor1-cs}.

\begin{corollary}\label{cor3}
Let the form of $f$ in \eqref{cs} be given by \eqref{cor1-cs}, and let $A$ and $B$ be as in Theorem~\ref{th1}. Assume in addition that the collection of eigenvalues with nonnegative real parts $\Lambda_+(A)$ of $A$ satisfies condition \eqref{lambda+}. If $d<n$, then the control system \eqref{cs} cannot be locally asymptotically stabilized by means of continuously differentiable stationary feedback laws.
\end{corollary}
\begin{proof}
Suppose on the contrary that the system under consideration can be locally asymptotically stabilized by means of continuously differentiable feedback laws. Then Theorem~\ref{th3} ensures that $f$ is linearly open around $(0,0)$ and hence it enjoys the openness property \eqref{open-prop} at this point. This contradicts the necessary condition for local asymptotical stabilization by smooth stationary feedback laws given by Brockett's theorem \cite[Theorem 1]{brockett83} and the first remark to it in \cite{brockett83} while thus justifying the conclusion of this corollary.
\end{proof}

\section{Asymptotic Stabilization of Discrete-Time Control Systems}\label{sec:disc}

In this section we implement the above techniques of variational analysis, together with known facts of control theory, for the case of nonlinear {\em discrete-time} control systems
\begin{align}
x_{k+1}=f(x_{k},u_{k}),\;k=0,1,\hdots,\label{dtcs}
\end{align}
where the index $k$ signifies discrete time. Our achievement here are more modest in comparison with the case of continuous-time control system in Section~\ref{sec2}. Namely, we establish only {\em sufficient} conditions (involving linear openness) for local asymptotic stabilization of \eqref{dtcs} by means of continuous stationary feedback laws. The possibility of deriving sufficient conditions for local exponential stabilization of \eqref{dtcs} as well as establishing advanced necessary conditions for both asymptotic and exponential versions of local stabilization by continuous and smooth stationary feedback laws constitutes important {\em open questions}.

To proceed, recall that the notion of {\em local asymptotic stabilization} of \eqref{dtcs} by means of continuous stationary feedback laws is formulated in a similar way to Definition~\ref{stab-def} with the replacement of the differential closed-loop system  \eqref{closed-loop} by its discrete counterpart
\begin{eqnarray*}
x_{k+1}=f\big(x_k,u(x_k)\big),\;k=0,1,\hdots.
\end{eqnarray*}
An interesting necessary condition in this direction, by using {\em smooth} stationary feedback laws, was obtained in \cite{lin-byr94} in terms of the openness property \eqref{open-prop} for the mapping $(x,u)\mapsto x-f(x,u)$.

The next theorem shows that the {\em linear openness} property of the mapping $f$, combined with appropriate relationships between its exact bound and the collection of eigenvalues which lie on or outside the unit circle of the matrix $A$ in the case of discrete-time control systems \eqref{dtcs}, ensures the local asymptotic stabilization of \eqref{dtcs} by continuous stationary feedback laws. To formulate the result, define the partial derivative matrices $A$ and $B$ as in \eqref{A} and consider the collection of {\em eigenvalues which lie on or outside the unit circle} of the (complex) spectrum $\Lambda(A)$ of $A$ in the case of \eqref{dtcs} given by
\begin{eqnarray}\label{lambda1}
\Lambda_1(A)\colonequals\big\{\lambda\in\Lambda(A)\big|\;\lambda\in\mathbb{C}\;\mbox{ \textnormal{with} }\;|\lambda|\ge 1\big\}.
\end{eqnarray}

\begin{theorem}\label{th4}
Let $f\colon\R^n\times\R^m\to\R^n$ be of class $C^{1}$ in a neighborhood of the equilibrium pair $(0,0)$, let the matrices $A$ and $B$ be defined in \eqref{A}, and let $\Lambda_1(A)\subset\R$. Assume in addition that:

{\bf(D)} $f$ is linearly open around $(0,0)$ with $\cov f(0,0)>\kappa$ for some positive constant $\kappa$ satisfying
\begin{eqnarray}\label{kappa-d}
\kappa>\eta_d\colonequals\sup_{\lambda\in\Lambda_1(A)}\lambda,
\end{eqnarray}
where the latter inequality is automatic if $\Lambda_1(A)=\emp$. Then the discrete-time nonlinear control system \eqref{dtcs} can be locally asymptotically stabilized by means of continuous stationary feedback laws.
\end{theorem}

\begin{proof}
We basically follow here the lines in the proof of Theorem~\ref{th1} for the continuous-time control system \eqref{cs} by using the same approach of variational analysis and the corresponding results of controllability and stabilization theories known for discrete-time control systems. For any $\nu>0$ and the discrete index $k=0,1,\hdots$ define the parametric family of vector functions
\begin{eqnarray*}
f_{\nu}(x_k,u_k)\colonequals f(x_k,u_k)+g_{\nu}(x_k,u_k)\;\mbox{ with }\;g_{\nu}(x_k,u_k)\colonequals-\nu x_k.
\end{eqnarray*}
Taking into account that $\lip g_{\nu}(0,0)=\nu$ and using the version of \cite[Theorem~4.25]{mor06a} for the stability of linear openness under Lipschitzian perturbations, we get the same lower estimate \eqref{cov-est} for the exact covering bound of $f$ around $(0,0)$ as in the proof of Theorem~\ref{th1}. Then remembering the choice of constants in condition {\bf (D)} imposed in this theorem, calculating $\eta_d$ as in \eqref{kappa-d} and employing the full rank condition for the Jacobian matrix $\nabla f_{\nu}(0,0)$ from the necessity part of Theorem~\ref{lg}, gives us the inequality
\begin{eqnarray*}
\cov f_{\lambda}(0,0)>\kappa-\lambda>0\;\mbox{ for all }\;\lambda\in\Lambda_1(A)
\end{eqnarray*}
with $\Lambda_1(A)$ defined in \eqref{lambda1}. This yields
\begin{eqnarray*}
\text{rank}\;\big[A-\lambda I\mid B\big]=n\;\mbox{ for all }\;\lambda\in\Lambda_1(A),
\end{eqnarray*}
which is the Hautus test for asymptotic controllability of the linearized discrete-time control system
\begin{eqnarray*}
x_{k+1}=Ax_k+Bu_k,\;k=0,1,\hdots,
\end{eqnarray*}
with $A$ and $B$ taken from \eqref{A}. Employing finally Hautus' stabilization theorem (see \cite[Theorem~4]{hautus70} for discrete-time control systems ensures the existence of a locally asymptotically stabilizing feedback law for \eqref{dtcs} and thus completes the proof of the result claimed in this theorem.
\end{proof}

Various consequences of Theorem~\ref{th4} can be derived similarly to those of Theorem~\ref{th1} in Section~\ref{sec2} for continuous-time control systems. Let us present just a few of them.

\begin{corollary}\label{4a}
Assume that in the framework of Theorem~{\rm\ref{th4}} the function $f$ is given by \eqref{lcs} with $\Lambda_{1}(A)\subset\R$, and let $g\colon\R^n\times\R^m\to\R^n$ be of class $C^{1}$ around the origin with $g(0,0)=0$, $\nabla_x g(0,0)=0$, and $\nabla_u g(0,0)=0$. Then the validity of condition {\bf(D)} from Theorem~{\rm\ref{th4}} ensures that the $($slightly$)$ perturbed nonlinear discrete-time control system
\begin{eqnarray*}
x_{k+1}=Ax_k+Bu_k+g(x_k,u_k),\;k=0,1,\hdots,
\end{eqnarray*}
can be locally exponentially stabilized by means of continuous stationary feedback laws.
\end{corollary}

\begin{proof}
Easily follows from Theorem~\ref{th4}.
\end{proof}

\begin{corollary}\label{4b}
Suppose that in the setting of Theorem~{\rm\ref{th4}} we have $\Lambda(A)\subset\R$ and that condition \eqref{kappa-d} in assumption {\bf(D)} is replaced by the following:
\begin{eqnarray*}
\kappa>\Tilde\eta_d\colonequals\max_{i\in\{1,\hdots,n\}}\big\{|\lambda_{i}|\;\big|\;\lambda_{i}\in\Lambda(A)\big\}.
\end{eqnarray*}
Then the nonlinear discrete-time control system \eqref{dtcs} can be locally asymptotically stabilized by means of continuous stationary feedback laws.
\end{corollary}

\begin{proof}
It is a clear consequence of Theorem~\ref{th4}; cf.\ also the arguments in the proof of Corollary~\ref{spectrum} for continuous-time control systems \eqref{cs} with the reference to \cite{coron07}.
\end{proof}

The last corollary of Theorem~\ref{th4} presented here addresses discrete-time control systems \eqref{dtcs} for which the whole spectrum $\Lambda(A)$ of the matrix $A$ in \eqref{A} consists only of zeros. This setting encompasses, in particular, an important case of discrete-time control systems where
\begin{eqnarray*}
f(x_{k},u_{k})\colonequals\sum_{i=1}^{n}g_{i}(x_{k})u_{ik}
\end{eqnarray*}
in \eqref{dtcs}
with the vectors $\{g_{i}(0)\}_{i=1}^{n}$ being linearly independent.

\begin{corollary}\label{4c}
Assume that in the setting of Theorem~{\rm\ref{th4}} we have $\Lambda(A)=\{0\}$. Then the linear openness of the vector function $f$ around $(0,0)$ ensures that the discrete-time control system \eqref{dtcs} can be locally asymptotically stabilized by means of continuous stationary feedback laws.
\end{corollary}

\begin{proof}
This clearly follows from Corollary~\ref{4b}.
\end{proof}

\section{Examples}\label{sec3}

In this section we present two examples illustrating the applications of the main result in Theorem~\ref{th1}. The first example deals with a two-dimensional control system while the second example deals with a more complicated three-dimensional one.

\begin{example}\label{ex3}
{\rm Consider the following continuous-time control system:
\begin{align}
\dot{x}_{1}&=x_{1}^{3}+x_{2},\label{cs_ex1_eq1}\\
\dot{x}_{2}&=u.\label{cs_ex1_eq2}
\end{align}
It is easy to verify that the pair $(0,0)$ is an equilibrium for \eqref{cs_ex1_eq1}--\eqref{cs_ex1_eq2}. Calculating the matrices $A$ and $B$ from \eqref{A} for this system at $(0,0)$ gives us
\begin{align}
A=\begin{bmatrix}
0 & 1\\
0 & 0
\end{bmatrix},\;
B=\begin{bmatrix}
0\\
1
\end{bmatrix}.\nonumber
\end{align}
By checking the rank condition \eqref{eqrem2} or directly by Definition~\ref{lin-open} we confirm that $f$ in \eqref{cs_ex1_eq1}--\eqref{cs_ex1_eq2} is linearly open around $(0,0)$. Furthermore, we can see that all the assumptions of Corollary~\ref{cor1a} are satisfied, and hence system \eqref{cs_ex1_eq1}--\eqref{cs_ex1_eq2} can be locally asymptotically stabilized by means of continuous stationary feedback laws. In fact, it is not hard to check that the feedback control function $u(x)\colonequals-x_{1}-x_{2}$ gives us one possible continuous stationary feedback law that locally asymptotically stabilizes system \eqref{cs_ex1_eq1}--\eqref{cs_ex1_eq2}.}
\end{example}

\begin{example}\label{ex4}
{\rm Consider next the following continuous-time control system:
\begin{align}
\dot{x}_{1}&=x_{1}^{3}+x_{3},\label{cs_ex2_eq1}\\
\dot{x}_{2}&=x_{1}+x_{3},\label{cs_ex2_eq2}\\
\dot{x}_{3}&=0.1x_{1}+x_{2}^{2}+u.\label{cs_ex2_eq3}
\end{align}
It is easy to verify that $(0,0)$ is an equilibrium pair for \eqref{cs_ex2_eq1}--\eqref{cs_ex2_eq3}. Thus we calculate by \eqref{A} the matrices
\begin{align}
A=\begin{bmatrix}
0 & 0 & 1\\
1 & 0 & 1\\
0.1 & 0 & 0
\end{bmatrix},\;
B=\begin{bmatrix}
0\\
0\\
1
\end{bmatrix}.\nonumber
\end{align}
The linear openness of $f$ in \eqref{cs_ex2_eq1}--\eqref{cs_ex2_eq3} can be deduced from the rank condition \eqref{eqrem2}, since
\begin{align}
\text{rank}\;\big[A\mid B\big]=3.\nonumber
\end{align}
Furthermore, we calculate that $\cov f(0,0)=0.6144$ and $\eta_c=0.3162$, and so assumption {\bf (C)} of Theorem~\ref{th1} is satisfied. Hence Theorem~\ref{th1} tells us that system \eqref{cs_ex2_eq1}--\eqref{cs_ex2_eq3} can be locally asymptotically stabilized by means of continuous stationary feedback laws. In fact, the feedback control function $u(x)\colonequals-x_{1}-x_{2}-x_{3}$ is one of the possible feedback laws that locally asymptotically stabilizes the control system \eqref{cs_ex2_eq1}--\eqref{cs_ex2_eq3}.}
\end{example}

\section{Concluding Remarks}\label{sec4}

The main trust of this paper is in suggesting a new approach of {\em variational analysis} to study feedback stabilization of nonlinear control systems. The emphasis of this approach is on applying the widely understood and comprehensively characterized {\em well-posedness} properties of single-valued and set-valued mappings to the issues under consideration instead of conventional topological methods related to degree theory, etc. Implementing these variational ideas in the context of {\em continuous-time} nonlinear control systems \eqref{cs} with smooth dynamics, we replace the topological {\em openness} property in the seminal Brockett's theorem by the variational {\em linear openness}. In this way, we are able to establish the {\em sufficiency} of the linear openness, together with a certain condition on its moduli, for the exponential stabilization of \eqref{cs} by means of continuous stationary feedback laws, while Brockett's theorem and its extensions assert the necessity of the openness property for asymptotic stabilization. We also derive new conditions ensuring the {\em necessity} of linear openness for both local exponential and asymptotic stabilization of \eqref{cs} by means of stationary continuous as well as smooth feedback laws. Some (but by far not all) counterparts of the obtained results are established by this approach for asymptotic feedback stabilization of nonlinear {\em discrete-time} control systems \eqref{dtcs}.

We believe that the suggested variational approach has strong potential to the study of feedback stabilization and related issues for a much broader class of systems and topics in comparison with those considered in this paper. A brief discussion on the available machinery of variational analysis and generalized differential given in Section~\ref{sec1a} sheds some light on appropriate tools and results. Among such topics we mention controllability and the minimum time function in the framework of \cite{col-nguyen13}, discontinuous feedback laws in the vein of \cite{ancona08,clss97,mor11}, extensions to the case of control systems governed by differential inclusions in both finite and infinite dimensions \cite{adly12,adly16}, feedback stabilization of partial differential control systems \cite{coron17,coron07,mor11}, etc.

\section*{Acknowledgments}
The work presented in this paper was carried out while the first author was a postdoctoral fellow at the Institute for Mathematics and its Applications (IMA) during the IMA's annual program on {\em Control Theory and its Applications}. Farhad Jafari also gratefully acknowledges the support and hospitality provided by IMA, where this work was initiated and where he was a visiting professor during the IMA's annual program on {\em Control Theory and its Applications}. The research of Boris Mordukhovich was partly supported by the US National Science Foundation under grant DMS-$1512846$, by the US Air Force Office of Scientific Research under grant $15$RT$0462$, and by the Ministry of Education and Science of the Russian Federation (Agreement number 02.a03.21.0008). This research was completed during Mordukhovich's stay at the University of Padova (February-March 2017), and he gratefully acknowledges fruitful discussions with and the warm hospitality provided by Prof.\ Giovanni Colombo.

\end{document}